\documentclass{amsart}
\usepackage{amsfonts, amssymb, amsmath, eucal, verbatim, amsthm, amscd, enumerate,mathabx}
\usepackage{setspace}
\usepackage{hyperref}

\newtheorem{theorem}{Theorem}[section]
\newtheorem{corollary}[theorem]{Corollary}
\newtheorem{lemma}[theorem]{Lemma}
\newtheorem{proposition}[theorem]{Proposition}

\newtheorem{thmx}{Theorem}

\theoremstyle{definition}

\theoremstyle{remark}

\newtheorem*{remarks}{Remarks}

\newtheorem*{acknowledgements}{Acknowledgements}

\usepackage{tikz}
\usepackage[toc,page]{appendix}

\numberwithin{equation}{section}
\def\R{{\mathbb R}}
\def\Z{{\mathbb Z}}

\def\d{{\mathrm{d}	}	}

\newcommand{\qrnorm}[2]{{L^{#1}_tL^{#2}_x}}

\author{Neal Bez}
\address[Neal Bez]{Department of Mathematics, Graduate School of Science and Engineering,
	Saitama University, Saitama 338-8570, Japan}
\email{nealbez@mail.saitama-u.ac.jp}

\author{Jayson Cunanan}
\address[Jayson Cunanan]{Department of Mathematics, Graduate School of Science and Engineering,
	Saitama University, Saitama 338-8570, Japan}
\email{jcunanan@mail.saitama-u.ac.jp}

\author{Sanghyuk Lee}
\address[Sanghyuk Lee]{Department of Mathematical Sciences and RIM, Seoul National University, Seoul 151-747, Korea}
\email{shklee@snu.ac.kr}


\begin{document}
	
\begin{abstract}
Strong-type inhomogeneous Strichartz estimates are shown to be false for the wave equation outside the so-called acceptable region. On a critical line where the acceptability condition marginally fails, we prove substitute estimates with a weak-type norm in the temporal variable. We achieve this by establishing such weak-type inhomogeneous Strichartz estimates in an abstract setting. The application to the wave equation rests on a slightly stronger form of the standard dispersive estimate in terms of certain Besov spaces.
\end{abstract}

	\date{\today}
	
	\title{Inhomogeneous Strichartz estimates in some critical cases}
	\maketitle 
	
\section{Introduction}

Consider the following Cauchy problems for the inhomogeneous wave equation
\[(\mathrm{W}) \,\, \begin{cases}
(\partial _t^2 -\Delta)u(t,x) = F(t,x),\quad \quad (t,x)\in\mathbb{R}^{1+d}, d\ge2\\
u(0,\cdot)=f,\quad \partial_t u(0,\cdot)=g,\\
\end{cases}\]
and the inhomogeneous Schr\"odinger equation
 \[(\mathrm{S}) \,\,	\begin{cases}
 (i\partial _t +\Delta)u(t,x) = F(t,x),\quad \quad (t,x)\in\mathbb{R}^{1+d}, d\ge1\\
 u(0,\cdot)=f.\\
 \end{cases}\]
By Duhamel's principle, solutions $u_\mathrm{W}$ and $u_\mathrm{S}$ to $(\mathrm{W})$ and $(\mathrm{S})$, respectively, can be written in the form
 \[u_\mathrm{W}(t,x)=\cos(t\sqrt{-\Delta})f(x)+\frac{\sin(t \sqrt{-\Delta})}{\sqrt{-\Delta}}g(x)-\int_0^t \frac{\sin((t-s)\sqrt{-\Delta})}{\sqrt{-\Delta}} F(s,\cdot)(x) \, \d s.
 \] and
 \[u_\mathrm{S}(t,x)=e^{it\Delta}f(x) - i\int_0^t e^{i(t-s)\Delta}F(s,\cdot)(x) \, \d s.
 \]
When $F$ is identically zero, we obtain the solution of the homogeneous problem, and when the initial data is set to zero we obtain the solution of the inhomogeneous problem.
 
An extremely useful family of estimates, known as \textit{Strichartz estimates}, quantify the size and decay of the solutions of evolution equations such as $(\mathrm{W})$ and $(\mathrm{S})$, typically through the use of mixed-norm spaces $L^q_tL^r_x(\R\times \R^d)$. It is natural to divide such estimates into \emph{homogeneous} Strichartz estimates
\begin{equation*}
\| U(t)f\|_{L^q_tL^r_x} \lesssim \|f\|_\mathcal{H}
\end{equation*}
and \emph{inhomogeneous} Strichartz estimates
\begin{equation*}
\bigg\|\int_{s<t} U(t)U^*(s)F(s,\cdot)\, \d s	\bigg\|_{L^q_tL^r_x} \lesssim \|	F	\|_{L^{\tilde{q}'}_tL^{\tilde{r}'}_x}
\end{equation*}
for an appropriate operator $U(t)$, Hilbert space $\mathcal{H}$, and exponent pairs $(q,r)$ and $(\tilde{q},\tilde{r})$. In the case of the wave and Schr\"odinger equations, it is standard to take $\mathcal{H}$ to be $L^2$ or, more generally, a homogeneous Sobolev space $\dot{H}^s$.

To a large extent, the Strichartz estimates for the wave and Schr\"odinger equations can be considered in a unified manner using the following abstract setting. Let $X$ be a measure space and $\mathcal{H}$ be a Hilbert space. For each $t\in \R$, suppose we have an operator $U(t): \mathcal{H} \rightarrow L^2(X)$ which satisfies the following energy estimate
\begin{equation}\label{e:energy}
	\|U(t)f\|_{L^2(X)}\lesssim\|f\|_{\mathcal{H}}
\end{equation}
and, for some $\sigma > 0$, the dispersive estimate
\begin{equation}\label{e:disp}
	\|U(t)U^*(s)g\|_{L^{\infty}(X)}\lesssim	|t-s|^{-\sigma}\|g\|_{L^1(X)}.
\end{equation}
We say that the exponent pair $(q,r)$ is \emph{sharp $\sigma$-admissible} if
\[
2\le q,r\le \infty, \quad\quad \frac{1}{q} =\sigma\bigg(\frac{1}{2}-\frac{1}{r}\bigg),	\quad\quad	\mathrm{and} \quad\quad		(q,r,\sigma)\neq(2,\infty,1).
\]

\begin{thmx}[Keel--Tao, \cite{KTao}]\label{KTao}
Suppose $U(t)$ satisfies $\eqref{e:energy}$ and $\eqref{e:disp}$. Then, the homogenous Strichartz estimate
\begin{equation}\label{e:homo}
		\|U(t)f\|_{L^q_tL^r_x(\mathbb{R} \times X)}\lesssim \|f\|_{\mathcal{H}}
\end{equation}
and the inhomogeneous Strichartz estimate
\begin{equation}\label{e:inhomo}
		\bigg\|\int_{s<t} U(t)U^*(s)F(s,\cdot)\, \d s	\bigg\|_{L^q_tL^r_x(\mathbb{R} \times X)} \lesssim \|	F	\|_{L^{\tilde{q}'}_tL^{\tilde{r}'}_x(\mathbb{R} \times X)}
\end{equation}
both hold for all sharp $\sigma$-$admissible$ pairs $(q,r)$ and $(\tilde q,\tilde r)$.
\end{thmx}

For the Schr\"odinger equation, we take $U(t) = e^{it\Delta}$ and $X = \mathbb{R}^d$ with $d \geq 1$, in which case \eqref{e:energy} holds with $\mathcal{H}=L^2(\R^d)$ and \eqref{e:disp} holds with $\sigma = \frac{d}{2}$. In this case, Theorem \ref{KTao} leads to a complete characterisation of the exponents $(q,r) \in [2,\infty] \times [2,\infty]$ for which the homogeneous Strichartz estimate
\begin{equation} \label{e:hom_S}
\| e^{it\Delta} f\|_{L^q_tL^r_x} \lesssim \|f\|_2
\end{equation}
holds; i.e. $(q,r)$ is sharp $\frac{d}{2}$-admissible.

For the wave equation, we take $X = \mathbb{R}^d$ with $d \geq 2$, and a standard approach is to consider the homogeneous estimates in terms of the one-sided and frequency-localised propagator $U(t) = e^{it\sqrt{-\Delta}} \chi_0(\sqrt{-\Delta})$, where $\chi_0$ is a bump function supported away from the origin. In this case, \eqref{e:energy} holds with $\mathcal{H}=L^2(\R^d)$ and \eqref{e:disp} holds with $\sigma = \frac{d-1}{2}$. In fact, thanks to the frequency-localisation, a stronger dispersive estimate holds with $(1 + |t-s|)^{-\sigma}$ on the right-hand side and a slight extension of Theorem \ref{KTao} (proved in \cite{KTao}) leads to a complete characterisation of the exponents $(q,r) \in [2,\infty] \times [2,\infty]$ for which the (frequency-localised) homogeneous Strichartz estimate
\begin{equation*}
\|  e^{it\sqrt{-\Delta}} \chi_0(\sqrt{-\Delta}) f\|_{L^q_tL^r_x} \lesssim \|f\|_2
\end{equation*}
holds; i.e. $\frac{1}{q} \leq \sigma(\frac{1}{2}-\frac{1}{r})$ and $(q,r,\sigma)\neq (2,\infty,1)$, with $\sigma = \frac{d-1}{2}$. When $r \neq \infty$, one can remove the frequency localisation by an application of the Littlewood--Paley inequality to give homogeneous Strichartz estimates of the form
\begin{equation} \label{e:hom_W}
\|  e^{it\sqrt{-\Delta}} f\|_{L^q_tL^r_x} \lesssim \|f\|_{\dot{H}^s},
\end{equation}
where $s = d(\frac{1}{2} - \frac{1}{r}) - \frac{1}{q}$. It is important for this application of the Littlewood--Paley inequality that $q,r \in [2,\infty)$. The case $r = \infty$ presents subtleties; for example, although \eqref{e:hom_S} holds when $(q,r,d) = (4,\infty,1)$, the analogous estimate  \eqref{e:hom_W} when $(q,r,d) = (4,\infty,2)$ was shown to fail  by Fang--Wang \cite{FangWang}. We also refer the reader to \cite{M-S} and \cite{GLNY} for the failure of \eqref{e:hom_W} when $(q,r) = (2,\infty)$ for $d=3$ and $d \geq 4$, respectively.

A number of papers have been crucial to the development of Strichartz estimates. A significant contribution in \cite{KTao} was to establish the endpoint case $(q,r) = (2,\frac{2\sigma}{\sigma - 1})$ for the homogeneous estimate \eqref{e:homo} when $\sigma > 1$. We refer the reader to \cite{CazWei,FangWang,GiVe,KTao,Str,TKat,Yaj} and their references within for further discussion.

The problem of determining all possible inhomogeneous Strichartz estimates \eqref{e:inhomo} for the wave and Schr\"odinger equations is still open. Foschi \cite{Fos} and Vilela \cite{Vil}, following the scheme of Keel and Tao \cite{KTao}, independently obtained a wider range of Lebesgue exponents than those given in Theorem \ref{KTao}. Foschi stated his results in the abstract framework above, while Vilela's statement is particular to the Schr\"odinger equation.

\begin{figure}
	\begin{center}
		\tiny
		\begin{tikzpicture}[scale=1.5]
		\draw [<->] (0,2.5) node (yaxis) [left] {$\frac{1}{q}$}
		|- (2.5/2,0) node (xaxis) [right] {$\frac{1}{r}$};
		\node [left]at (0,1) {$\frac{1}{2}$};
		
		\node[below] at (0,0) {$0$};
		\node[below] at (1,0) {$\frac{1}{2}$};
		
		\node[below] at (1/2,0) {\tiny $\frac{\sigma-1}{2\sigma}$};
		\node[below] at (0,0) {0};
		\node[below] at (1,0) {$\frac{1}{2}$};
		
		\draw[dotted] (0,2)--(.5,2);
		\draw[dotted] (1/2,0)--(1/2,1)--(0,1);
		
		\shadedraw[very thin]  (.5,2)--(0,2)--(0,0)--(1,0);
		\draw[dashed] (.5,2)--(1,0);
		\draw[thick]  (1,0)--(.5,1);
		\end{tikzpicture}
	\end{center}
	\caption{Admissible and acceptable exponents.}  \label{f:1}
\end{figure}
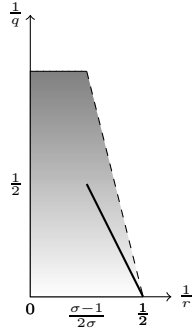

In the case of the Schr\"odinger equation, for the inhomogeneous Strichartz estimate \eqref{e:inhomo} to hold, the scaling condition
\begin{equation} \label{e:scaling_S}
\frac{1}{q} + \frac{1}{\tilde{q}} = \frac{d}{2}\bigg(1 - \frac{1}{r} - \frac{1}{\tilde{r}} \bigg)
\end{equation}
is easily seen to be necessary. Also, in \cite{Fos,Vil}, several further conditions were shown to be necessary for \eqref{e:inhomo} to hold, including
\begin{equation} \label{e:acceptable_S}
\frac{1}{q} < d\bigg(\frac{1}{2} - \frac{1}{r}\bigg)
\end{equation}
and its counterpart for $(\tilde{q},\tilde{r})$ (later, this condition will appear in the definition of $\frac{d}{2}$-acceptability) and
\begin{equation} \label{e:qnecessary_S}
\frac{1}{q} \leq \frac{1}{\tilde{q}'}.
\end{equation}
In \cite{Fos,Vil} it was shown that \eqref{e:inhomo} holds under \eqref{e:scaling_S} and \eqref{e:acceptable_S}, and further assumptions which differ depending on whether $d=1$, $d=2$ or $d \geq 3$; we refer the reader to these papers for the precise statement of their result (see also \cite{Koh, KohSeo} for certain improvements to the range of exponents). Although \eqref{e:inhomo} fails when $(q,r)$ are such that \eqref{e:acceptable_S} marginally fails, that is
$
\frac{1}{q} = d(\frac{1}{2} - \frac{1}{r}),
$
the third author and Seo proved in \cite{LSeo} that certain weak-type estimates of the form
\begin{equation}\label{e:Lseo}
	\bigg\|	\int_{s<t} e^{i(t-s)\Delta}F(s,\cdot)\, \d s	\bigg\|_{L_t^{q,\infty}L^r_x}\lesssim \|F\|_{L_t^{\tilde q'}L_x^{\tilde r'}}
\end{equation}
are true in such a critical case; see Corollary \ref{c:weakS} below for a precise statement.

As suggested above in our discussion of the homogeneous case, the inhomogeneous estimates for the wave equation (without any frequency localisation assumption) follow less easily. Taggart \cite{Tag} succeeded in obtaining a result for the wave equation, analogous to the result obtained by Foschi and Vilela for the Schr\"odinger equation, by making use of a slightly stronger dispersive estimate
\begin{equation} \label{e:dispersiveW_Besov}
\|	e^{it\sqrt{-\Delta}}	f\|_{\dot{B}^{-\rho}_{\infty,2}}\lesssim |t|^{-(d-1)/2}\|	f\|_{\dot{B}^{\rho}_{1,2}}
\end{equation}
for data in the Besov space $\dot{B}^{\rho}_{1,2}$, where $\rho = \frac{d+1}{4}$. This falls outside the scope of the abstract result of Foschi, whose basic assumptions were \eqref{e:energy} and \eqref{e:disp}, as in Theorem \ref{KTao}. Taggart's inhomogeneous estimates for the wave equation were thus built on establishing a generalisation of Foschi's abstract result which accommodated \eqref{e:dispersiveW_Besov}; roughly speaking, one replaces the Lebesgue spaces where the spatial variable lies with appropriate Banach spaces. This approach actually gives rise to stronger inhomogeneous Strichartz estimates of the form
\begin{equation}\label{e:waveTaggart}
\bigg\|	\int_{s<t} e^{i(t-s)\sqrt{-\Delta}}F(s,\cdot) \, \d s	\bigg\|_{L_t^{q}\dot{B}_{r,2}^{-\gamma} 	} \lesssim \|F\|_{L^{\tilde q'}\dot{B}_{\tilde r',2}^{\tilde\gamma} }
\end{equation}
and embeddings involving Besov spaces and Lebesgue spaces yield estimates with $L^q_tL^r_x$ norms. Notice that this approach avoids explicit use of Littlewood--Paley theory. This has a major advantage since it is possible for the temporal exponent $q$ to fall below $2$ and this seems to prevent a straightforward approach based on the standard $L^1 \to L^\infty$ dispersive estimate and Littlewood--Paley theory from succeeding (see the end of Section \ref{section:proofs} for further remarks in this direction).

Our main goal in this note is to consider inhomogeneous Strichartz estimates in certain critical cases. To describe this, following Foschi \cite{Fos}, for $\sigma > 0$ we say that the exponent pair $(q,r)$ is $\sigma$-$acceptable$ if
\[
1\le q< \infty, 2\le r\le\infty, \quad\quad \frac{1}{q} < 2\sigma\bigg(\frac{1}{2}-\frac{1}{r}\bigg),	\quad\quad \mathrm{or}	 \quad\quad (q,r)=(\infty,2).
\]
Figure \ref{f:1} depicts the range of sharp $\sigma$-admissible and $\sigma$-acceptable exponents. 

We remark that the condition $\frac{1}{q} < 2\sigma(\frac{1}{2}-\frac{1}{r})$ appearing in the definition of $\sigma$-acceptability arose in the discussion above in \eqref{e:acceptable_S} as a necessary condition for the inhomogeneous estimates for the Schr\"odinger equation. The analogous condition for the wave equation seems not to have been established before, so we provide a proof in Section \ref{section:necessary} that $\frac{1}{q} < (d-1)(\frac{1}{2}-\frac{1}{r})$ is necessary in the case of inhomogeneous estimates for the wave equation. 

In our main result below, we consider the critical case $\frac{1}{q} = 2\sigma(\frac{1}{2}-\frac{1}{r})$ and obtain a generalisation of the weak-type estimate \eqref{e:Lseo} to an abstract setting analogous to the framework used by Taggart \cite{Tag}. As an application, we establish such weak-type estimates for the wave equation in the critical case $\frac{1}{q} = (d-1)(\frac{1}{2}-\frac{1}{r})$.

We now state the main result in this note. 
\begin{theorem}\label{t:qweak}
Let $\sigma>0$, $\mathcal{H}$ be a Hilbert space and $\mathcal{B}_0, \mathcal{B}_1$ be Banach spaces. For each $t\in \R$, suppose we have an operator $U(t): \mathcal{H} \rightarrow \mathcal{B}_0^*$ which satisfies the following estimates
\begin{align}
	\|U(t)f\|_{\mathcal{B}_0^*}&\lesssim\|f\|_{\mathcal{H}}, \label{e:energybanach}\\
		\|U(t)U^*(s)g\|_{\mathcal{B}_1^{\ast}}&\lesssim	|t-s|^{-\sigma}\|g\|_{\mathcal{B}_1}. \label{e:dispbanach}
\end{align}
Suppose one of the following conditions holds.
\begin{itemize}
	\item $\sigma<1$ and  $\theta,\tilde\theta\in[0,1]$;
	\item $\sigma=1$ and $\theta,\tilde\theta\in[0,1)$;
	\item $\sigma>1$ and
	\[ \frac{\sigma-1}{\sigma}(1-\theta)\le(1-\tilde\theta)	,\quad  \frac{\sigma-1}{\sigma}(1-\tilde\theta)\le(1-\theta)	.	\]
\end{itemize}
Then for all $q$ and $\tilde q$ verifying
\begin{equation}\label{e:qweakscale}
\sigma\theta=\frac{1}{q}<\frac{1}{\tilde q'}=1-\sigma\bigg(\frac{\tilde\theta}{2}-\frac{\theta}{2}\bigg),\quad \theta\le\tilde\theta,
\end{equation}
the following estimate
\begin{equation}\label{e:qweak}
\bigg\|\int_{s<t} U(t)U^*(s)F(s,\cdot)\, \d s	\bigg\|_{L_t^{q,\infty}\mathcal{B}^*_{\theta}} \lesssim \|	F	\|_{L_t^{\tilde{q}'}\mathcal{B}_{\tilde\theta}}
\end{equation}
holds, where $\mathcal{B}_{\theta}$ denotes the real interpolation space $(\mathcal{B}_0,\mathcal{B}_1)_{\theta,2}$. 
\end{theorem}

Observe that by taking $\mathcal{H} = L^2$, $(\mathcal{B}_0, \mathcal{B}_1) = (L^2, L^1)$, and $U(t)=e^{it\Delta}$, along with the Lorentz space embedding $L^{r'}\subset \mathcal{B}_{\theta}=L^{r',2}$ ($r = \frac{2}{1-\theta}$), then we recover the result in \cite{LSeo} with the addition of the case $d=2.$ 
\begin{corollary}\label{c:weakS}
Suppose one of the following conditions holds.
	\begin{itemize}
		\item $d=1$ and  $r,\tilde r\in(2,\infty]$;
		\item $d=2$ and $r,\tilde r\in(2,\infty)$;
		\item $d\ge 3$ and
		\[ \frac{d-2}{r}\le\frac{d}{\tilde r}	,\quad \frac{d-2}{\tilde r}\le\frac{d}{ r}	.	\]
	\end{itemize} Then for all $q$ and $\tilde q$ verifying
	\[d\bigg(\frac{1}{2}-\frac{1}{r}\bigg)=\frac{1}{q}<\frac{1}{\tilde q'}=1-\frac{d}{2}\bigg(\frac{1}{r}-\frac{1}{\tilde r}\bigg),\quad \frac{1}{\tilde r}\le\frac{1}{r},
	\]
	the following estimate
	\begin{equation}
\bigg\|	\int_{s<t} e^{i(t-s)\Delta}F(s,\cdot) \, \d s	\bigg\|_{L_t^{q,\infty}L^r_x}\lesssim \|F\|_{L_t^{\tilde q'}L_x^{\tilde r'}}
\end{equation}
	holds.
\end{corollary}
By duality, Theorem \ref{t:qweak} also obtains estimates for the cases $q=\infty$ or $\tilde q=\infty$ which are excluded in the definition of $\sigma$-acceptable exponents used in the statements of previous results such as \cite{CazWei,TKat,Fos,Vil,Koh,KohSeo}. Indeed, take $0\le \theta=\tilde \theta< 1/\sigma$ in \eqref{e:qweakscale} with $\tilde q=\infty$, in which case we obtain
\begin{align}
\bigg\|\int_{s<t} U(t)U^*(s)F(s,\cdot)\, \d s	\bigg\|_{L_t^{q,\infty}\mathcal{B}^*_{\theta}}  \lesssim \|	F	\|_{L_t^{1}\mathcal{B}_{\theta}} \\
\bigg\|\int_{s<t} U(t)U^*(s)F(s,\cdot)\, \d s	\bigg\|_{L_t^{\infty}\mathcal{B}^*_{\theta}}  \lesssim \|	F	\|_{L_t^{q',1}\mathcal{B}_{\theta}} 
\end{align}
for $q > 1$. 

As discussed already, the level of generality of Theorem \ref{t:qweak} accommodates dispersive estimates in terms of Besov spaces and thus allows one to avoid use of Littlewood--Paley theory to extend frequency-localised estimates to frequency-global estimates. A direct application of Theorem \ref{t:qweak} to the wave equation is the following.
\begin{corollary}\label{c:weakwave} Suppose  one of the following conditions holds.
	\begin{itemize}
		\item $d=2$ and  $r,\tilde r\in(2,\infty]$;
		\item $d=3$ and $r,\tilde r\in(2,\infty)$;
		\item $d\ge 4$ and
		\[ \frac{d-3}{r}\le\frac{d-1}{\tilde r}	,\quad \frac{d-3}{\tilde r}\le\frac{d-1}{ r}	.	\]
	\end{itemize} Then for all $q$ and $\tilde q$ verifying
	\[(d-1)\bigg(\frac{1}{2}-\frac{1}{r}\bigg)=\frac{1}{q}<\frac{1}{\tilde q'}=1-\frac{d-1}{2}\bigg(\frac{1}{r}-\frac{1}{\tilde r}\bigg),
\]	and
	\[\gamma=\frac{d+1}{2}\bigg(\frac{1}{2}-\frac{1}{ r}\bigg)\quad \tilde \gamma=\frac{d+1}{2}\bigg(\frac{1}{2}-\frac{1}{ \tilde r}\bigg) .
\]
	the following estimate
	\begin{equation}
	\bigg\|	\int_{s<t} e^{i(t-s)\sqrt{-\Delta}}F(s,\cdot)\, \d s	\bigg\|_{L_t^{q,\infty}\dot{B}_{r,2}^{-\gamma} 	} \lesssim \|F\|_{L^{\tilde q'}\dot{B}_{\tilde r',2}^{\tilde\gamma} } \label{e:w1}
	\end{equation}
	holds. Therefore, 
	\begin{equation}
\bigg\|	\int_{s<t} e^{i(t-s)\sqrt{-\Delta}}F(s,\cdot)\, \d s	\bigg\|_{L^{q,\infty}_tL^r_x}\lesssim \|(\sqrt{-\Delta})^{\gamma+\tilde \gamma}F\|_{L^{\tilde q'}_tL^{\tilde r'}_x} \label{e:w2}
	\end{equation}
	also holds whenever $r,\tilde r<\infty$ .
\end{corollary}	
In particular, take $\gamma+\tilde\gamma=1,q=r$ and $\tilde q=\tilde r$, then Corollary \ref{c:weakwave} states that if $d\ge2,$
\[\frac{1}{\tilde q'}-\frac{1}{q}=\frac{2}{d+1}
\]
and
\[ 
\frac{d-1}{2d}=\frac{1}{q}<\frac{1}{\tilde q'}=\frac{2}{d+1}-\frac{d-1}{2d},
\]
then 
\begin{equation}\label{e:w3}
\| u\|_{L^{q,\infty}_{t}L^q_x}\lesssim\|F\|_{L^{\tilde q'}_{t,x}}
\end{equation}
holds for any solution $u$ to $(\mathrm{W})$ with zero initial data.

Note that \eqref{e:w2} follows from \eqref{e:w1} by the embeddings
\[
\dot{B}_{r,2}^{\rho}\subseteq \dot{H}^{\rho}_r  \quad\quad \mathrm{and} \quad \quad \dot{H}^{\rho}_{r'} \subseteq \dot{B}_{r',2}^{\rho} 
\]
for $2\le r<\infty$ and $\rho\in \R$, and the fact that $(\sqrt{-\Delta})^{\alpha}$ is an isomorphism from $\dot{B}^{\gamma}_{r,2}$ to $\dot{B}_{r,2}^{\gamma-\alpha}$. Here, $\dot{H}^\rho_r$ denotes the homogeneous Sobolev space whose norm is defined by $		\|	f\|_{\dot{H}^\rho_r} =\|(\sqrt{-\Delta})^{\rho}f	\|_{L^p}$.

For \eqref{e:w2}, we note that $L^{q,\infty}$ cannot be improved to a smaller Lorentz space $L^{q,p}$ with $p < \infty$. We prove this assertion in Section \ref{section:necessary} (the corresponding claim for the Schr\"odinger equation and \eqref{e:Lseo} was explained in \cite{LSeo}). For further remarks on Theorem \ref{t:qweak} and its connection to other results in the literature, we refer the reader to Section \ref{section:necessary}.


\section{Proofs of the main results} \label{section:proofs}
The following time-local estimates were obtained in \cite[Proposition 4.1]{Tag} by Taggart and were used to generalise the abstract result of Foschi \cite{Fos} from $L^r_X$ spaces to an interpolation family of abstract Banach spaces $\mathcal{B}_{\theta}=(\mathcal{B}_0,\mathcal{B}_1)_{\theta,2}$; we refer the reader to \cite{BerghLofstrom} for details of real interpolation spaces. As noted above, we will need this level of generality to be able to prove \eqref{e:w2} by avoiding Littlewood--Paley theory and thus handle the cases where $q < 2$.
\begin{lemma} \cite[Proposition 4.1]{Tag} \label{l:localbanach}
	Let $\sigma>0$ and assume $U(t)$ satisfies \eqref{e:energybanach} and \eqref{e:dispbanach}. Suppose one of the following conditions holds.
	\begin{itemize}
		\item $\sigma<1$ and  $\theta,\tilde\theta\in[0,1]$;
		\item $\sigma=1$ and $\theta,\tilde\theta\in[0,1)$;
		\item $\sigma>1$ and
		\[ \frac{\sigma-1}{\sigma}(1-\theta)\le(1-\tilde\theta)	,\quad  \frac{\sigma-1}{\sigma}(1-\tilde\theta)\le(1-\theta)	.	\]
	\end{itemize}
	Suppose $q$ and $\tilde q$ verify either
\[0\le\frac{1}{q}\le\frac{1}{\tilde q'}\le1-\sigma\bigg(\frac{\tilde\theta}{2}-\frac{\theta}{2}\bigg),\quad \theta\le\tilde\theta,
\]
or
\[-\sigma\bigg(\frac{\tilde\theta}{2}-\frac{\theta}{2}\bigg)\le\frac{1}{q}\le\frac{1}{\tilde q'}\le1,\quad \tilde\theta\le\theta.
\]
Then, for all $j \in \mathbb{Z}$, the estimate
	\begin{equation}\label{e:timedelaybanach}
\bigg	\|	\int_{|t-s|\sim 2^j} U(t)U^*(s)F(s,\cdot)\, \d s	\bigg\|_{L^{q}_t\mathcal{B}^*_{\theta}} \lesssim 2^{-j\beta(\theta,\tilde{\theta},q,\tilde{q})}\|	F	\|_{L^{\tilde q'}_t\mathcal{B}_{\tilde \theta}}
	\end{equation}
	holds, where
	\begin{equation}
	\beta(\theta,\tilde{\theta},q,\tilde{q})=\frac{\sigma}{2}(\theta+\tilde\theta)-\frac{1}{\tilde{q}}-\frac{1}{q}.
	\end{equation}
\end{lemma}
To prove Theorem \ref{t:qweak}, we will need the following summation lemma which allows us to sum \eqref{e:timedelaybanach} in $j\in\Z$ to obtain \eqref{e:qweak}.
\begin{lemma}\label{l:btrick}
	Let $\epsilon_1,\epsilon_2>0$ and $1\le q_1,q_2<\infty.$ For a given Banach space $\mathcal{B}$, suppose $\{f_j\}_{j\in\Z}	$ is a collection of $\mathcal{B}$-valued functions defined on $ \R $ such that 
	\begin{equation}\label{e:bounds}
	\|	f_j\|_{L^{q_1}(\mathcal{B}) }\le M_12^{\epsilon_1j} \text{ and }  \|	f_j\|_{L^{q_2}(\mathcal{B}) }\le M_22^{-\epsilon_2j}.
	\end{equation}
	Then
	\[ 
	\bigg \|\sum_j f_j \bigg\|_{L^{q,\infty}(\mathcal{B}) }\le CM_1^{\theta}M_2^{1-\theta}
	\]
	where $ \theta=\epsilon_2/(\epsilon_1+\epsilon_2) $, $1/q=\theta/q_1+(1-\theta)/q_2$ and $C$ is a constant depending on $\epsilon_1,\epsilon_2,q_1$ and $q_2.$
\end{lemma}
The above lemma is due to Bourgain \cite{Bourgain}; see also \cite{LSeo} or \cite{LSeo2}. For completeness, we provide a short proof.
\begin{proof}[Proof of Lemma \ref{l:btrick}]
Let $\lambda > 0$ and $N\in\mathbb{Z}$ which will be chosen at the end of the proof. Then, by Chebyshev's inequality and the triangle inequality, we have
	\[
	\bigg|\bigg\{t: \bigg\|\sum_{j=-\infty}^\infty f_j\bigg\|_{\mathcal{B}} >\lambda\bigg\}\bigg| \lesssim \lambda^{-q_1} I_1 + \lambda^{-q_2} I_2
	\]
	where
	\begin{align*}
	I_1 & := \int_\mathbb{R} \bigg(\sum_{j=-\infty}^N\|f_j(\cdot,t)\|_{\mathcal{B}}\bigg)^{q_1}\, \d t \\
	I_2 & := \int_\mathbb{R} \bigg(\sum_{j=N+1}^\infty \|f_j(\cdot,t)\|_{\mathcal{B}}\bigg)^{q_2}\, \d t.
	\end{align*}

	By the triangle inequality and \eqref{e:bounds}, we obtain
	$
	I_1^{1/q_1} \lesssim M_12^{\epsilon_1N} 
	$
	and
	$
	I_2^{1/q_2} \lesssim M_22^{-\epsilon_2N}.
	$
	Putting these bounds together and optimising in the choice of $N$ yields
	\[
	\bigg|\bigg\{t: \bigg\|\sum_{j=-\infty}^\infty f_j\bigg\|_{\mathcal{B}} >\lambda\bigg\}\bigg|^{1/q} \lesssim M_1^{\theta}M_2^{1-\theta}\lambda^{-1} 
	\]
	as claimed.
\end{proof}

\subsection{Proof of Theorem \ref{t:qweak}}

Define the operator $T_j$ by
\[T_jF=\int_{|t-s|\sim2^j}U(t)U^*(s)F(s,\cdot)\, \d s.
\]
Then,
\[\int^t_{-\infty}U(t)U^*(s)F(s,\cdot)\, \d s=\sum_{j\in\Z}T_jF.
\]
Fix $\theta,\tilde \theta,q,\tilde q$ satisfying \eqref{e:qweakscale}. Then we verify that $\frac{\sigma}{2}(\theta+\tilde\theta)-\frac{1}{\tilde{q}}-\frac{1}{q}=0$. Now we choose $q_1,q_2$ given by
\[	\frac{1}{q_1}=\frac{1}{q}+\delta,\quad \frac{1}{q_2}=\frac{1}{q}-\delta
\]
for some $\delta>0$ sufficiently small so that the conditions of Lemma \ref{l:localbanach} holds for pairs $(q_i,\theta)$ and $(\tilde q,\tilde\theta), i=1,2.$ Such $\delta$ exists since we assume $1/q<1/\tilde q'$. 
Thus, by  Lemma \ref{l:localbanach} we have
\[\|	T_jF	\|_{L^{q_i}_t\mathcal{B}^*_{\theta}}\lesssim 2^{j(\frac{\sigma}{2}(\theta+\tilde\theta)-\frac{1}{\tilde{q}}-\frac{1}{q_i})}	\|		F\|_{L^{\tilde q_i}_t\mathcal{B}_{\tilde\theta}}
\]and we note that
\[\frac{\sigma}{2}(\theta+\tilde\theta)-\frac{1}{\tilde{q}}-\frac{1}{q_1}<0<\frac{\sigma}{2}(\theta+\tilde\theta)-\frac{1}{\tilde{q}}-\frac{1}{q_2}.
\]
Applying Lemma \ref{l:btrick} with $\epsilon_i= |	\frac{\sigma}{2}(\theta+\tilde\theta)-\frac{1}{\tilde{q}}-\frac{1}{q_i}		|$ gives the desired estimate \eqref{e:qweak}. \qed

\subsection{Proof of Corollary \ref{c:weakwave}}
We make use of the dispersive estimate \eqref{e:dispersiveW_Besov}, which we restate here
\begin{equation} \label{e:dispersiveW_Besov_again}
\|	e^{it\sqrt{-\Delta}}	f\|_{\dot{B}^{-\rho}_{\infty,2}}\lesssim |t|^{-(d-1)/2}\|	f\|_{\dot{B}^{\rho}_{1,2}},
\end{equation}
where $\rho = \frac{d+1}{4}$. Here, we recall the Besov norm is given by
\[
\|f\|_{\dot{B}^\rho_{r,s}} = \bigg( \sum_{j \in \mathbb{Z}} 2^{\rho s j} \| f * \varphi_j \|_{L^r(\mathbb{R}^d)}^s \bigg)^{1/s}
\]
where $\widehat{\varphi}_j = \widehat{\varphi}(2^{-j} \cdot)$, $\widehat{\varphi}\in C^\infty_0[\frac{1}{2},2]$ and $\sum_{j\in\Z}\widehat{\varphi}_j =1$.

Although the estimate \eqref{e:dispersiveW_Besov_again} can be found in \cite{GiVe}, we give a proof for completeness. Given $\varphi$ as above, the standard dispersive estimate
\begin{equation} \label{e:dispersive_W_standard}
\sup_{x\in\R^d}\bigg|	\int_{\R^d}	e^{i(x\cdot\xi+t|\xi|)}\widehat{\varphi}(\xi)\, \d\xi		\bigg|\lesssim |t|^{-(d-1)/2}
\end{equation}
and an elementary rescaling argument gives
\[
\| e^{it\sqrt{-\Delta}}\varphi_j	\|_{L^\infty(\R^d)} \lesssim |t|^{-(d-1)/2}2^{j(d+1)/2}.
\]
This easily yields
\[
2^{-j(d+1)/4} \| \varphi_j * e^{it\sqrt{-\Delta}}f	\|_{L^\infty(\R^d)} \lesssim |t|^{-(d-1)/2} 2^{j(d+1)/4} \| \widetilde{\varphi}_j * f\|_{L^1(\R^d)}
\]
where $\widetilde{\varphi}$ is similar to $\varphi$, chosen such that $\varphi_j * \widetilde{\varphi}_j = \varphi_j$. Taking the $\ell^2$ norm of each sequence gives \eqref{e:dispersiveW_Besov_again}.

\begin{proof}[Proof of Corollary \ref{c:weakwave}]
For the operator $U(t) = e^{it\sqrt{-\Delta}}$, we apply Theorem \ref{t:qweak} with $\mathcal{H} = L^2$ and $(\mathcal{B}_0,\mathcal{B}_1) = (L^2,\dot{B}_{1,2}^{\rho})$, where $\rho = \frac{d+1}{4}$. The estimate \eqref{e:w1} now follows by making use of the interpolation space identity (see, for example, \cite[Lemma 8.1]{Tag})
\[
(\dot{B}^{\rho_0}_{r_0,2},\dot{B}^{\rho_1}_{r_1,2})_{\theta,2} = \dot{B}^\rho_{(r,2),2},
\]
where $r_0,r_1 \in [1,\infty)$, $\rho_0,\rho_1 \in \mathbb{R}$ with $\rho_0 \neq \rho_1$, and $\theta \in (0,1)$. Here,the norm for $\dot{B}^\rho_{(r,2),2}$ is the same as for $\dot{B}^\rho_{r,2}$, except the $L^r$ norm is replaced by the Lorentz space norm $L^{r,2}$. Thus, $\mathcal{B}_\theta = \dot{B}^\gamma_{(r',2),2} \supset \dot{B}^{\gamma}_{r',2}$, where $\gamma = \frac{d+1}{4} \theta$ and $\frac{1}{r'} = \frac{1+\theta}{2}$, and \eqref{e:w1} follows.
\end{proof}

\begin{remarks}
(1) Using the dispersive estimate \eqref{e:disp}, Foschi proved Lemma \ref{l:localbanach} for the case $\mathcal{B}_{\theta}=(L^2,L^1)_{\theta,2}.$ From the standard dispersive estimate for the wave equation \eqref{e:dispersive_W_standard}, this yields a frequency-localised version of \eqref{e:w2}. Littlewood--Paley theory appears to only allow us to remove this localisation in the restricted range $q, \tilde{q}\ge 2$.

(2) Related to the remark above, we finish this section with an observation about strong-type inhomogeneous estimates for the wave equation. Lemma \ref{l:localbanach} in the case $\mathcal{B}_{\theta}=(L^2,L^1)_{\theta,2}$ (due to Foschi) generates frequency-localised strong-type inhomogeneous estimates. To circumvent the issues with applying Littlewood--Paley theory to remove the localisation, rather than employing the strong form of Lemma \ref{l:localbanach} as carried out by Taggart, one may proceed with a bilinear interpolation argument in the $q\tilde q$-plane. This generates estimates of the form
	\begin{equation}\label{e:waveBil}
	\bigg\|	\int_{s<t} e^{i(t-s)\sqrt{-\Delta}}F(s,\cdot) \, \d s	\bigg\|_{L_t^{q}L^r_x	} \lesssim \|(\sqrt{-\Delta})^\gamma F\|_{L^{\tilde q'}_tL^{\tilde r' }_x}.
	\end{equation}
	and rests on the standard dispersive estimate rather than \eqref{e:dispersiveW_Besov_again}.
\end{remarks}

\section{Necessary conditions and further remarks} \label{section:necessary}

In this section, $C$ denotes a positive constant which may change from line to line.

\subsection{A necessary condition for the wave equation}

\begin{proposition}\label{p: necwave}
	Let $d\ge2,1\le q,\tilde q<\infty,2\le r,\tilde r\le\infty$ and $1\le p<\infty$. If the estimate
		\begin{equation}\label{alphawave}
	\bigg	\|\int_{\R} e^{i(t-s)\sqrt{-\Delta}}F(s,\cdot)\, \d s	\bigg\|_{\qrnorm{q,p}{r}} \lesssim \|	F	\|_{\qrnorm{\tilde{q}'}{\tilde{r}'}}
		\end{equation}
holds for all $F$ whose (spatial) Fourier support is compactly supported away from zero, then
\[\frac{1}{q}<(d-1)\bigg(\frac{1}{2}-\frac{1}{r}\bigg).
\]
\end{proposition}
This proposition shows that inhomogeneous Strichartz estimates for the wave equation generically fail unless the condition $\frac{1}{q}<(d-1)(\frac{1}{2}-\frac{1}{r})$ holds. By standard arguments, it follows that the corresponding estimate to \eqref{alphawave} with the $s$-integral taken over $\int_{-\infty}^t$ also fails unless $\frac{1}{q}<(d-1)(\frac{1}{2}-\frac{1}{r})$.
\begin{proof}
	Choose $\widehat{F}(t,\xi)=\psi(t)\varphi(|\xi|)$ where $\psi$ is an even $C^\infty(\R)$ function supported in $(-c_0,c_0)$ for $c_0 > 0$ sufficiently small and $\varphi\in C_0^{\infty}[\frac{1}{2},2]$. Setting $\tilde{\varphi}(r)=r^{\frac{d-1}{2}}\varphi(r)$, by using polar coordinates, we have 
	\[	\int_{\R} e^{i(t-s)\sqrt{-\Delta}}F(s,\cdot)(x)\, \d s =C\left(|x|^{-\frac{d-1}{2}}I(x,t) + |x|^{-\frac{d-1}{2}}II(x,t)\right),
	\]
	where
	\begin{align*}
	I(x,t)=&\int_{\R}\int_0^{\infty} e^{i(t-s)r}\tilde\varphi(r)\cos(r|x|-\tfrac{\pi(d-1)}{4})\, \d r\ \psi(s)\, \d s,\\
	II(x,t)=&\int_{\R}\int_0^{\infty} e^{i(t-s)r}r^{\frac{d-3}{2}}\varphi(r)E(r|x|)\, \d r\ \psi(s)\, \d s.
	\end{align*}
	Here, the error function $E$ arises from the asymptotic estimate
	\begin{align*} 
	\int_{\mathbb{S}^{d-1}} e^{ix \cdot \theta} \, \mathrm{d}\theta & =C|x|^{-\frac{d-2}{2}}J_{\frac{d-2}{2}}(|x|) \\
	& =C|x|^{-\frac{d-1}{2}}\bigg(\cos(r|x|-\tfrac{\pi(d-1)}{4})+E(|x|)	\bigg),
	\end{align*}
	where $E(\rho)\lesssim \rho^{-1}$ for $\rho \geq 1$ (see, for example, \cite{SteinWeiss}). Then 
	\[ I(x,t)=\frac{1}{2}\left(e^{-i\pi(d-1)/4} I_-(x,t)+e^{i\pi(d-1)/4} I_+(x,t)
	\right)
	\]
	where
	\begin{align*}
	I_{\pm}(x,t)&=\int_{\R}\int_0^{\infty} e^{-i[\pm|x|-(t-s)]r}\tilde \varphi(r)\, \d r\ \psi(s)\, \d s\\
	&=\int_{\R} \widehat{\tilde{\varphi}}(\pm|x|-(t-s)) \psi(s)\, \d s\\
	&=\vartheta\ast \psi(\pm|x|-t)
	\end{align*}
	and we have set $\vartheta:=\widehat{\tilde{\varphi}}$. Since $\vartheta(0)>0,$ it follows that $\vartheta\ast \psi(0)>0$ for a sufficiently small choice of $c_0.$ By continuity, there exists $\delta_0>0$
such that $\vartheta\ast \psi(t)>0$ for $|t|\le\delta_0.$ Thus,
\[|I_+(x,t)|\gtrsim 1 \quad\text{for}\quad\big||x|-t	\big|\le\delta_0.
\]
Also, $|\vartheta(t)|\lesssim 1/(1+|t|)$ for all $t\in\R$ by an integration by parts argument and since $\varphi\in C^\infty_0(\R)$.
This implies
\[|I_-(x,t)| \lesssim \frac{1}{1+||x|+t|}\lesssim \frac{1}{t},
\]
and therefore $|I(x,t)|\gtrsim1$, for $\big||x|-t	\big|\le\delta_0$ and $t\ge100.$
Also, $|II(x,t)| \lesssim  \int_{r\sim1} |E(r|x|)|\, \d r\lesssim1/|x|,$
so
\begin{align*}
\bigg |\int_{\R} e^{i(t-s)\sqrt{-\Delta}}F(s,\cdot)\, \d s\bigg|&\gtrsim |x|^{-\frac{d-1}{2}}\big(I-II\big)\\
&\gtrsim t^{-\frac{d-1}{2}}
\end{align*}
if $\big||x|-t	\big|\le\delta_0$ and $t\ge100.$ So, for $t\ge100,$
\begin{align*}
\bigg \|\int_{\R} e^{i(t-s)\sqrt{-\Delta}}F(s,\cdot)\, \d s\bigg\|_{L^r_x}&\ge \bigg \|\int_{\R} e^{i(t-s)\sqrt{-\Delta}}F(s,\cdot)\, \d s\bigg\|_{L^r_x(||x|-t|\le\delta_0)}\\
&\gtrsim t^{(d-1)(\frac{1}{r}-\frac{1}{2})}.
\end{align*}
If follows that
\[\bigg \|\int_{\R} e^{i(t-s)\sqrt{-\Delta}}F(s,\cdot)\, \d s\bigg\|_{L^{q,p}_tL^r_x}\gtrsim \bigg(\int_0^1\lambda^{p(1-\frac{1}{\alpha q})-1}\,\d \lambda			\bigg)^{1/p}
\]
where $\alpha=(d-1)(\frac{1}{2}-\frac{1}{r})$. On the other hand, clearly $\|F\|_{L_t^{\tilde q'}{L_x^{\tilde r'}}}<\infty. $
Therefore, the condition 
$\frac{1}{q}< \alpha$
is necessary for \eqref{alphawave}, as claimed.
\end{proof}

\subsection{Some limitations}
Recall the crucial assumption 
\begin{equation} \label{e:crucialcondition}
\sigma\theta=\frac{1}{q}<\frac{1}{\tilde q'}=1-\sigma\bigg(\frac{\tilde\theta}{2}-\frac{\theta}{2}\bigg)
\end{equation}
in Theorem \ref{t:qweak}. By translation invariance considerations, the assumption $\frac{1}{q} \leq \frac{1}{\tilde q'}$ is natural, and thus we have not addressed the case $q = \tilde{q}'$ in this note. Our approach as it stands does not permit this case since we rely on Bourgain's summation trick (Lemma \ref{l:btrick}) and this seems to constrain us to a situation where $q$ can freely move in some open range. 

The strict inequality in condition \eqref{e:crucialcondition} also means we are unable to address the case $q=1$ with our approach. Related to this, the estimate
\begin{equation} \label{e:Beceanu_stronger}
\int_\mathbb{R} \|e^{it\Delta}f\|_{L^{6,\infty}(\mathbb{R}^3)} \, \mathrm{d}t \lesssim \|f\|_{L^{6/5,1}(\mathbb{R}^3)}
\end{equation}
was claimed to be true in \cite[Proposition 1.2]{Be}. By Minkowski's inequality, this implies the inhomogeneous Strichartz estimate 
\begin{equation} \label{e:Beceanu_inhomo}
\bigg\| \int_{s<t} e^{i(t-s)\Delta}F(s,\cdot)\, \mathrm{d}s\bigg \|_{L^1_tL^{6,\infty}_x(\mathbb{R} \times \R^3)} \lesssim\|F\|_{L^{1}_t L^{6/5,1}_x(\mathbb{R} \times \mathbb{R}^3)},
\end{equation}
which seems to provide a particular case with $q = \tilde{q}' = 1$ on the critical line $\frac{1}{q} = d(\frac{1}{2}-\frac{1}{r})$, albeit in restricted weak-type form with respect to the spatial variable (see \cite[Theorem 1.3]{Be}). Unfortunately, it seems that \eqref{e:Beceanu_stronger} and \eqref{e:Beceanu_inhomo} are both false. Through personal communications, M. Beceanu informed us that a variant form of \eqref{e:Beceanu_stronger} is still possible \cite{BFS}.

More generally, the estimate
\begin{equation} \label{e:Beceanu_inhomo_d}
\bigg\| \int_\mathbb{R} e^{i(t-s)\Delta}F(s,\cdot) \, \mathrm{d}s \bigg\|_{L^1_t L^{\frac{2d}{d-2},\infty}_x} \lesssim \| F \|_{L^1_tL^{\frac{2d}{d-2},1}_x}
\end{equation}
is false for all $d \geq 3$, and hence the stronger estimate
\begin{equation*}
\int_{\mathbb{R}} \|e^{it\Delta}f\|_{L^{\frac{2d}{d-2},\infty}} \, \mathrm{d}t \lesssim \|f\|_{L^{\frac{2d}{d-2},1}} 
\end{equation*}
is also false for all $d \geq 3$. To see the failure of \eqref{e:Beceanu_inhomo_d}, consider the function $F$ given by
\[
F(t,x) = \frac{1}{1 + 4t^2} \exp(-\tfrac{1}{2}|x|^2) 
\]
so that 
\[
\widehat{F}(\tau,\xi) = C \exp(-\tfrac{1}{2}(|\xi|^2 + |\tau|))
\]
for some constant $C$. In this case, we have
\begin{align*}
\int_{\mathbb{R}^d} e^{i(x \cdot \xi - t|\xi|^2)}  \widehat{F}(-|\xi|^2,\xi) \, \mathrm{d}\xi & = C \int_{\mathbb{R}^d} e^{ix \cdot \xi} e^{-(1 + it)|\xi|^2} \, \mathrm{d}\xi \\
& = C \frac{1}{(1 + it)^{d/2}} \exp\bigg(-\frac{|x|^2}{4(1+it)}\bigg).
\end{align*}
Therefore, if $r \in [1,\infty]$, it follows that
\begin{align*}
\bigg\| \int_\mathbb{R} e^{i(t-s)\Delta}F(s,\cdot) \, \mathrm{d}s \bigg\|_{L^{r,\infty}_x} \sim |t|^{d(\frac{1}{r} - \frac{1}{2})}
\end{align*}
holds for all $|t| \gtrsim 1$. Thus, in the case $r = \frac{2d}{d-2}$ with $d \geq 3$, we have
\begin{align*}
\bigg\| \int_\mathbb{R} e^{i(t-s)\Delta}F(s,\cdot) \, \mathrm{d}s \bigg\|_{L^{\frac{2d}{d-2},\infty}_x} \sim \frac{1}{|t|}
\end{align*}
for all $|t| \geq 1$ and hence
\begin{align*}
\bigg\| \int_\mathbb{R} e^{i(t-s)\Delta}F(s,\cdot) \, \mathrm{d}s \bigg\|_{L^1_t L^{\frac{2d}{d-2},\infty}_x} = \infty.
\end{align*}
On the other hand, it is that clear that
$
\| F \|_{L^{a}_tL^{b,1}_x} < \infty
$
for any $a,b \in [1,\infty)$.

It is clear that the above counterexample does not rule out the possibility of \eqref{e:Beceanu_inhomo_d} being true with the weak-type space $L^{1,\infty}_t$ on the left-hand side. Related to this, we remark that estimates of the shape $L^1_t L^{r'}_x \to L^{1,\infty}_t L^r_x$ are of substantial interest since they would lead to an improvement in the result of Foschi and Vilela \cite{Fos, Vil} in the case $q = \tilde{q}'$.

\begin{acknowledgements}
This work was supported by JSPS Grant-in-Aid for Young Scientists A no. 16H05995 (Bez), a JSPS Postdoctoral Research Fellowship no. 18F18020 (Cunanan), and NRF-2015R1A4A1041675 (Lee). 
\end{acknowledgements}


\begin{thebibliography}{MMM}
	
	\bibitem{Be} M. Beceanu, \textit{New estimates for a time-dependent Schr\"odinger equation}, Duke Math. J. \textbf{159} (2011), 417--477.		
	
	\bibitem{BFS} M. Beceanu, J. Fr\"ohlich, A. Soffer, \textit{Liouville's equation with random potential in Schatten-von Neumann classes}, preprint.
	
	\bibitem{BerghLofstrom} J. Bergh, J. L\"ofstr\"om, \textit{Interpolation Spaces: An Introduction}, Springer--Verlag, New York, 1976.
	
	\bibitem{Bourgain} J. Bourgain, \textit{Estimations de certaines functions maximales}, C. R. Acad. Sci. Paris \textbf{310} (1985) 499--502.
	
	\bibitem{CazWei} T. Cazenave, F.B. Weissler, \textit{The Cauchy problem for the nonlinear Schr\"odinger equation in $H^1$}, Manuscripta Math. \textbf{61} (1988), 477--494.  
	
	
	
	\bibitem{FangWang} D. Fang, C. Wang, \textit{Some remarks on homogeneous estimates for wave equation}, Nonlinear Anal. \textbf{65} (2006), 697--706.
	
	\bibitem{Fos} D. Foschi, \textit{Inhomogeneous Strichartz estimates}, J. Hyperbolic Differ. Equ. \textbf{2} (2005), 1--24.
	
	\bibitem{GiVe} J. Ginebre, G. Velo, \textit{Generalized Strichartz inequalities for the wave equation}, J. Func. Anal. \textbf{133} (1995), 50--68.
	
	\bibitem{GLNY} Z. Guo, J. Li, K. Nakanishi, L. Yan, \textit{On the boundary Strichartz estimates for wave
	and Schr\"odinger equation}, J. Differential Equations \textbf265 (2018), 5656--5675.
		
	
	\bibitem{TKat} T. Kato, \textit{An $L^{q,r}$-theory for nonlinear Schr\"odinger equations, Spectral and scattering theory and applications}, Adv. Stud. Pure Math. \textbf{23}, Math. Soc. Japan, Tokyo (1994), 223--238.
	
	\bibitem{KTao} M. Keel, T. Tao, \textit{Endpoint Strichartz estimates}, Amer. J. Math. \textbf{120} (1998), 955--980.
	
	\bibitem{Koh} Y. Koh, \textit{Improved inhomogeneous Strichartz estimates for the Schr\"odinger equation}, J. Math. Anal. Appl. \textbf{373}
	(2011), 147--160.
	
	\bibitem{KohSeo} Y. Koh, I. Seo, \textit{Inhomogeneous Strichartz estimates for Schr\"odinger's equation}, J. Math. Anal. Appl. \textbf{442}
	(2016), 715--725.
	
	\bibitem{LSeo} S. Lee, I. Seo, \textit{A note on unique continuation for the Schr\"odinger equation}, J. Math. Anal. Appl. \textbf{389}
	(2012), 461--468.
	
	\bibitem{LSeo2} S. Lee, I. Seo, \textit{Sharp bounds for multiplier operators of negative indices associated with degenerate curves}, Math. Z. \textbf{267} (2011), 291--323.
	

\bibitem{M-S} S. J. Montgomery-Smith, \textit{Time decay for the bounded mean oscillation of solutions of the Schr\"odinger and wave equations}, Duke Math. J. \textbf{91} (1998), 393--408.	
	
	
	

\bibitem{SteinWeiss} E. M. Stein, G. Weiss, \textit{Introduction to Fourier Analysis on Euclidean Spaces},
Princeton Univ. Press, Princeton, N. J., 1971.	
	
	\bibitem{Str} R.S. Strichartz, \textit{Restriction of Fourier transform to quadratic surfaces and decay of solutions of wave equations}, Duke
	Math. J. \textbf{44} (1977), 705--774.
	
	\bibitem{Tag} R. Taggart, \textit{Inhomogeneous Strichartz estimates}, Forum Math. \textbf{22} (2010), 825--853.
	
	
	\bibitem{Vil} M. C. Vilela, \textit{Inhomogeneous Strichartz estimates for the  Schr\"odinger equation}, Trans. Amer. Math. Soc. \textbf{359} (2007),
	2123--2136.
	
	
	\bibitem{Yaj} K. Yajima, \textit{Existence of solutions for Schr\"odinger evolution equations}, Comm. Math. Phys.
	\textbf{110} (1987), 415--426.
	
\end{thebibliography}
\end{document}